\documentclass[11pt]{article}

\usepackage{amssymb}
\usepackage{amsmath}

\newtheorem{theorem}{Theorem}
\newtheorem{proposition}{Proposition}

\newenvironment{proof}{\mbox{\sc Proof:}}{$\Box$}
\newenvironment{proof*}{\mbox{\sc Proof:}\hspace{0.3 em}{\bf (*)}\hspace{0.3 em}}{$\Box$}

\begin{document}

\title{Characterization of the Ito Integral}
\author{Lars Tyge Nielsen \\ Department of Mathematics \\ Columbia University}
\date{December 2018}
\maketitle

\begin{abstract}
This paper provides an existence-and-uniqueness theorem characterizing the stochastic integral with respect to a Wiener process. The integral is represented as a mapping from the space of measurable and adapted pathwise locally integrable processes to the space of continuous adapted processes. It is characterized in terms of two properties: (1) how the stochastic integrals of simple processes are calculated and (2) how these integrals converge in probability when the time integrals of the squared integrands converge in probability.
\end{abstract}

\section{Introduction}

The Ito integral, or the stochastic integral with respect to a Wiener process, is used extensively in financial engineering and other areas of applied mathematics.  
Its construction is somewhat complicated, which in applied work often motivates a less-than-rigorous exposition or a limitation to integrands in $\mathcal{H}^{2}$, those that are square integrable with respect to the product of the probability measure and the Lebesgue measure on the time axis.

Defining the Ito integral only for $\mathcal{H}^{2}$-integrands, as in Shreve \cite[2004]{Shreve:04}, leads to further complications, even in applied work. Here are some examples.

(1) Ito's Lemma tells us that if $W$ is a Wiener process and $f$ is a twice continuously differentiable function, then $f(W)$ is the sum of a time integral and an Ito integral. Additional assumptions are needed to guarantee that that the Ito integrand is in $\mathcal{H}^{2}$. 

(2) The Martingale Representation Theorem, 
Rogers and Williams \cite[1987, Theorem 36.5]{Rogers-Williams:87},
says that a martingale with respect to the Wiener filtration is a stochastic integral. However, the integrand is not in 
$\mathcal{H}^{2}$ unless the martingale is square integrable.

(3) Dudley's Representation Theorem, Dudley \cite[1977]{Dudley:77},
says that if a random variable is measurable with respect to a sigma-algebra $\tilde{\mathcal{F}}_{T}^{W}$
from the augmented Wiener filtration, for some $T > 0$, then it arises as an Ito integral up to time $T$.
Only if the variable is square integrable can the integrand be chosen from $\mathcal{H}^{2}$.

This paper provides an easy way to define the Ito integral in the general case of measurable adapted pathwise locally integrable processes.
We present an existence-and-uniqueness theorem that characterizes the integral in terms of simple properties and
thus pins it down as a well-defined mathematical object.

One main benefit of the theorem is that someone who is interested primarily in applications -- for example in finance -- can skip the proof, and thus skip the entire cumbersome construction, and yet proceed with a completely mathematically rigorous understanding of the Ito integral.

The integral is described as a mapping from the space of measurable and adapted pathwise locally integrable processes to the space of continuous adapted processes. It is characterized in terms of how integrals of simple processes are calculated and how such integrals converge in probability when the time integrals of the squared integrands converge in probability.

\section{The Existence-and-Uniqueness Theorem}

All processes in this paper are understood to be one-dimensional.

A \emph{setup} is a quintuple $(\Omega, \mathcal{F},P,F,W)$ consisting of 
\begin{itemize}
\item
a complete probability space $(\Omega, \mathcal{F},P)$,
\item
an augmented filtration $F = (\mathcal{F}_{t})_{t \in [0,\infty)}$, and 
\item
a standard Wiener process $W$ with respect to $F$. 
\end{itemize}
The filtration does not need to be right-continuous.

We shall assume that a specific setup has been chosen. 

Let
$\mathcal{L}^{2}$ denote the set of measurable and adapted processes that are pathwise locally square integrable with probability one. In other words, measurable and adapted processes $b$ such that for every $t \in [0,\infty)$,
\[ P \left( \int_{0}^{t} b^{2} \, ds \right) < \infty \]

These processes will be our integrands. We shall define and characterize the stochastic integral of processes in $\mathcal{L}^{2}$ with respect the Wiener process $W$.

We could alternatively assume the integrands to be progressive processes, since every measurable and adapted process has a progressive modification. See Kaden and Potthoff \cite[2004, Theorem~1]{Kaden-Potthoff:04},
Ondrej\'{a}t and Seidler \cite[2013, Theorem~0.1]{Ondrejat-Seidler:13}.

A process $b$ is \emph{simple}\index{process!simple}\index{simple process} if 
there exists a finite
sequence of times $0 = t_{0} < t_{1} < \cdots < t_{n}$ such that 
\[ b = b(t_{0}) 1_{\{0\}} +
\sum_{i = 1}^{n}b(t_{i})1_{(t_{i-1},t_{i}]} \]
and such that
$b(t_{0})$ is measurable with respect to $\mathcal{F}_{0}$ and for each $i = 1, \ldots ,n$, $b(t_{i})$ is measurable with respect to $\mathcal{F}_{t_{i-1}}$.

A simple process is measurable and adapted.

Let
$\mathcal{H}^{2} \subset \mathcal{L}^{2}$ denote the set of measurable and adapted locally square integrable processes. In other words, measurable and adapted processes $b$ such that for every $t \in [0,\infty)$,
\[ \mathbb{E} \int_{0}^{t} b^{2} \, ds < \infty \]
A simple process $b$ is in $\mathcal{H}^{2}$ if and only if $\mathbb{E}b(t)^{2} < \infty$ for every $t > 0$.

Let $\mathcal{C}$ denote the set of adapted continuous processes, and identify two such processes if they are indistinguishable.

A mapping
$I: \mathcal{L}^{2} \rightarrow \mathcal{C}$ will be called a \emph{stochastic integral mapping}
if it has the following two properties:
\begin{enumerate}
\item
\emph{Integrals of Simple Processes}: 
Let $b \in \mathcal{L}^{2}$ be simple, let $t \in [0,\infty)$, and let
$0 = t_{0} < t_{1} < \cdots < t_{n} \leq t$ 
be a finite sequence of times 
such that 
\[ b = b(t_{0}) 1_{\{0\}} +
\sum_{i = 1}^{n}b(t_{i})1_{(t_{i-1},t_{i}]} \]
Then
\[ I(b)(t) = \sum_{i=1}^{n}
b(t_{i})(W(t_{i})-W(t_{i-1})) \]
with probability one.
\item
\emph{Convergence in Probability}: Let $t \in [0,\infty)$. If $b \in \mathcal{L}^{2}$ and 
$(b_{n})$ is a sequence of simple processes in $\mathcal{H}^{2}$ such that
\[ \int_{0}^{t} (b - b_{n})^{2} \, ds \rightarrow 0 \]
in probability, then
\[ I(b_{n})(t) \rightarrow I(b)(t) \]
in probability.
\end{enumerate}

\begin{theorem}
\label{charstochint-t}
\marginpar{charstochint-t}
Given the setup $(\Omega,\mathcal{F},P, F,W)$, there exists a unique stochastic integral mapping $I: \mathcal{L}^{2} \rightarrow \mathcal{C}$.
\end{theorem}

The theorem defines the stochastic integral, which is usually written in the form
\[ I(b)(t) = \int_{0}^{t} b \, dW \]
for $b \in \mathcal{L}^{2}$ and $t \in [0,\infty)$.

\section{Proof of the Theorem}

The proof relies on approximation of processes in $\mathcal{L}^{2}$ by simple processes.

\begin{proposition}
\label{simpprocapp-p}
Let $b \in \mathcal{L}^{2}$. For each $t \in [0,\infty)$, 
there exists a sequence
$(b_{n})$ of simple processes in $\mathcal{H}^{2}$ such that $b_{n}1_{(t,\infty)} = 0$ for all $n$ and
\[ \int_{0}^{t} (b - b_{n})^{2} \, ds \rightarrow 0 \]
with probability one (and hence, in probability).
\end{proposition}
\begin{proof}
Liptser and Shiryaev \cite[2001, Lemma 4.5]{Liptser-Shiryaev:01}\index{Liptser, R. S.}\index{Shiryaev, A. N.}.
\end{proof}

\medskip
The existence statement in Theorem~\ref{charstochint-t} is amply proved in the literature, as we shall now document.

{\sc Proof of Theorem~\ref{charstochint-t} -- Existence}

An integral with properties 1 and 2 is constructed, and thus its existence is shown, for example in
Arnold \cite[1974]{Arnold:74}\index{Arnold, L.}
and Liptser and Shiryaev \cite[2001]{Liptser-Shiryaev:01}\index{Liptser, R. S.}\index{Shiryaev, A. N.}.
The construction is summarized in Nielsen \cite[1999]{Nielsen:99}.

First, the integral $I(b)(t)$ of a simple process is defined for each $t \in [0,\infty)$, by the formula in property 1. It is observed that $I(b)(t)$ is measurable with respect to $\mathcal{F}_{t}$. 
See Liptser and Shiryaev \cite[2001, pages 95--96]{Liptser-Shiryaev:01}\index{Liptser, R. S.}\index{Shiryaev, A. N.}.

Then the integral is defined for a general process $b \in \mathcal{L}^{2}$ by a procedure of approximation in probability. 
If $b$ is approximated in probability by a sequence $(b_{n})$ of simple process as in Proposition~\ref{simpprocapp-p}, then the sequence of stochastic integrals $I\left( b_{n}\right)(t)$ will converge in probability to some random variable which is unique with probability one and independent of the approximating sequence $(b_{n})$. This limiting random variable is defined to be the stochastic integral $I(b)(t)$ of $b$. See
Liptser and Shiryaev \cite[2001, pages 107--108]{Liptser-Shiryaev:01}\index{Liptser, R. S.}\index{Shiryaev, A. N.}.

This definition of the integral is obviously consistent with the formula in property 1.

The integral $I(b)(t)$ is measurable with respect to $\mathcal{F}_{t}$ because it is the probability limit of the integrals  $I\left(b_{n}\right)(t)$, which are measurable with respect to $\mathcal{F}_{t}$, and because $\mathcal{F}_{t}$ is augmented.

Finally, it is shown that the process $I(b)$ has a continuous modification.
See
Liptser and Shiryaev \cite[2001, pp.\ 108--109]{Liptser-Shiryaev:01}\index{Liptser, R. S.}\index{Shiryaev, A. N.}.
The continuous modification is unique up to indistinguishability, and it is adapted, because the filtration is augmented.

Property 2, convergence in probability, is shown in
Liptser and Shiryaev \cite[2001, pp.\ 107--108]{Liptser-Shiryaev:01}\index{Liptser, R. S.}\index{Shiryaev, A. N.}.

$\Box$

\newpage

{\sc Proof of Theorem~\ref{charstochint-t} -- Uniqueness}

Suppose $I$ and $J$ are stochastic integral mappings $\mathcal{L}^{2} \rightarrow \mathcal{C}$.

Let $b \in \mathcal{L}^{2}$.

Let $t \in [0,\infty)$. By Proposition~\ref{simpprocapp-p}, there exists a sequence
$(b_{n})$ of simple processes in $\mathcal{H}^{2}$ with $1_{(t,\infty)}b_{n} = 0$ for all $n$, such that
\[ \int_{0}^{t} (b - b_{n})^{2} \, ds \rightarrow 0 \]
in probability.
For each $n$,
there exists a finite sequence of times 
$0 = t_{0} < t_{1} < \cdots < t_{n} \leq t$ 
such that
\[ b_{n} = b_{n}(t_{0}) 1_{\{0\}} +
\sum_{i = 1}^{m}b_{n}(t_{i})1_{(t_{i-1},t_{i}]} \]
By the definition of a stochastic integral mapping, property 1, 
\[ I\left(b_{n}\right)(t) = \sum_{i=1}^{m}
b_{n}(t_{i})(W(t_{i})-W(t_{i-1})) = J\left(b_{n}\right)(t) \]
with probability one.

By the definition of a stochastic integral mapping, property 2, 
\[ I\left(b_{n}\right)(t) \rightarrow I(b)(t) \]
and
\[ J\left(b_{n}\right)(t) \rightarrow J(b)(t) \]
in probability. This implies that $I(b)(t) = J(b)(t)$ with probability one.

Since $I(b)$ and $J(b)$ are continuous and stochastically equivalent, they are indistinguishable.

$\Box$

\bibliographystyle{plain}
\bibliography{bookrefs,bookrefsadd,consolidated}

\end{document}